\documentclass[12pt,bezier]{article}
\usepackage{indentfirst}
\usepackage{changepage}
\usepackage{amssymb}
\usepackage{enumerate}
\usepackage{mathrsfs}
\usepackage{amsmath}
\usepackage{amsfonts,amsthm,amssymb}
\usepackage{amsfonts}
\usepackage{graphicx}
\usepackage{caption}
\usepackage{float}
\usepackage{lineno}
\usepackage{color}
\usepackage{changes}
\usepackage[numbers,sort&compress]{natbib}

\newtheorem{lem}{Lemma}[section]
\newtheorem{thm}[lem]{Theorem}

\newtheorem{defi}[lem]{Definition}
\newtheorem{con}[lem]{Conjecture}

\begin{document}
	\title{On the minimum degree of minimally $ t $-tough, claw-free graphs}
	
	\author{Hui Ma, Xiaomin Hu, Weihua Yang\footnote{Corresponding author. E-mail: ywh222@163.com; yangweihua@tyut.edu.cn}\\
		\\ \small Department of Mathematics, Taiyuan University of
		Technology,\\
		\small  Taiyuan Shanxi-030024,
		China\\
	}
	\date{}
	\maketitle
	
	{\small{\bf Abstract:} A graph $ G $ is minimally $ t $-tough if the toughness of $ G $ is $ t $ and deletion of any edge from $ G $ decreases its toughness. Katona et al. conjectured that the minimum degree of any minimally $ t $-tough graph is $ \lceil 2t\rceil $ and proved that the minimum degree of minimally $ \frac{1}2 $-tough and $ 1 $-tough, claw-free graphs is 1 and 2, respectively. We have show that every minimally $ 3/2 $-tough, claw-free graph has a vertex of degree of $ 3 $. In this paper, we give an upper bound on the minimum degree of minimally $t$-tough, claw-free graphs for $ t\geq 2 $.
		
		
		\section{Introduction}
		
		All graphs considered in this paper are finite, undirected, and without loops or multiple edges. Let $ G=\big(V(G), E(G)\big) $ be a graph. For any vertex $ v $ of $ G $, we denote the set of vertices of $ G $ adjacent to $ v $ by $ N_{G}(v) $ and the number of edges of $ G $ incident to $ v $ by $ d_{G}(v) $. Let $ N_{G}[u]=\left\{u\right\}\cup N_{G}(u) $ and $ \delta(G)=\min\left\{d_{G}(v) : v\in V(G)\right\} $. Let $ \kappa(G) $ denote the vertex connectivity of $ G $. For any vertex set $ S\subseteq V(G) $, we use $ G[S] $ and $ G-S $ to denote the subgraph of $ G $ induced by $ S $ and $ V(G)-S $, respectively, and $ w(G-S) $ to denote the number of components of $ G-S $. If $ S, T$ are subsets of $ V(G) $ or subgraphs of $ G $, let $ N_{S}(T) $ denote the set of vertices in $ S-T $ that are adjacent to some vertex of $ T $. The subgraph of $ G $ is said to be an atom, denoted by $ A $, if $ A=\min\left\{\lvert A_{i}\rvert : \lvert N_{G}(A_{i})\rvert=\kappa(G), w\big(G-N_{G}(A_{i})\big)\geq 2 \right\} $. 
		
		\begin{defi}
			The toughness of a graph $ G $, denoted by $ \tau(G) $, is defined as follows,
			$$ \tau(G)=\min\left\{\dfrac{\lvert S\rvert}{w(G-S)} : S\subseteq V(G), w(G-S)\geq 2\right\}$$
			if $ G $ is not complete, and $ \tau(G)=+\infty $ if $ G $ is complete.
			We say that a graph is $ t$-tough if its toughness is at least $ t $.
		\end{defi}
		
		Since Chv\'atal introduced a new graph invariant called toughness in 1973, many results concerning the relatin betweeen toughness and the existance of Hamilton cycles or factors have been obtained (see \cite{Lu, Enomoto, Kab, Nis}). To find a graph which is 2-tough but not hamiltonian-connected, Broersma introduced the definition of minimally $ t $-tough graphs (\cite{Broersma}).
		\begin{defi}
			A graph $ G $ is said to be minimally $ t $-tough if $ \tau(G)=t $ and $ \tau(G-e)<t $ for any $ e\in E(G) $.
		\end{defi}
		We note that it follows from the definition of toughness that a $t$-tough noncomplete graph is $2t$-connected. Thus the minimum degree of a $t$-tough noncomplete graph is at least $ \lceil 2t\rceil$. It is well-known that the minimum degree of every minimally $k$-connected graph is exactly $k$ (\cite{Mader}). Thus, the following is a natural conjecture.
		
		\begin{con}[Kriesell \cite{Kaiser}]
			Every minimally 1-tough graph has a vertex of degree 2.
		\end{con}
		
		We use $ n $ to denote the order of a graph. Katona et al. considered Kriesell's conjecture and show that every minimally 1-tough graph has a vertex of degree at most $ \frac{n}{3}+1 $. In addition, they proposed a generalized Kriesell's conjecture.
		
		\begin{con}[Generalized Kriesell's Conjecture \cite{Katona}]
			Every minimally t-tough graph has a vertex of degree $ \lceil 2t\rceil $.
		\end{con}	
		
		A claw-free graph is a graph containing no copy of  $ K_{1,3} $ as an  induced subgraph. There are two open conjectures in terms of connectivity and toughness of claw-free graphs. Matthews and  Sumner conjectured that every 2-tough, claw-free graph is hamiltonian (\cite{Matthews}), and Thomassen conjectured that every 4-connected line graph is hamiltonian (\cite{Thomassen}). In \cite{R} it is shown that the two conjectures are equivalent. 
		
		By a theorem of Matthews and Sumner \cite{Matthews}, the toughness of a claw-free graph is exactly half of its connectivity. Katona et al. gave a completely characterization of minimally $ \frac{1}2 $-tough and $ 1 $-tough, claw-free graphs in \cite{Gyula,Katona}. Recently, we have show that Generalized Kriesell's Conjecture is true for minimally $ 3/2 $-tough, claw-free graphs in \cite{Ma}.
		
		\begin{thm}[\cite{Katona}]
			The class of minimally $ 1/2 $-tough, claw-free graphs consists of exactly those graphs that can be built up in the following way.
			\begin{enumerate}[1.]
				\item Take a tree $ T $ with maximum degree at most $ 3 $ where the set of vertices of degree $ 1 $ and $ 3 $ together form an independent set.
				\item Now delete every veretx of degree $ 3 $, but connect its $ 3 $ neighbors with a triangle.
			\end{enumerate}
		\end{thm}

		\begin{thm}[\cite{Gyula}]
			If $ G $ is a minimally 1-tough, claw-free graph of order $ n\geq 4 $, the $ G=C_{4} $.
		\end{thm}
		
		\begin{thm}[\cite{Ma}]
			Every minimally $ 3/2 $-tough, claw-free graph has a vertex of degree $ 3 $,
		\end{thm}
		
		In this paper, we gave an upper bound of the minimum degree of any minimally $ t $-tough, claw-free graphs, where $ t\geq 2 $. The proof is in Section 2.
		
		\begin{thm}\label{result 12}
			Let $ t\geq 2 $. Every minimally $t$-tough, claw-free graph has a vertex of degree at most $ \big\lceil \frac{10t-5}{3}\big \rceil$.
		\end{thm}

		
		
		\section{Proof of Theorem \ref{result 12}}
		
		\begin{thm}[\cite{Matthews}]\label{result 8}
			If $ G $ is a noncomplete claw-free graph, then $ 2\tau(G)=\kappa(G) $.
		\end{thm}
		
		From the definition of minimally $ t $-tough graphs, we can obtain the following lemma.
		
		\begin{lem}[\cite{Katona}]\label{result 1}
			Let $ t $ be a positive rational number and $ G $ a minimally $ t $-tough graph. For every edge $ e $ of $ G $,
			\begin{enumerate}
				\item the edge $ e $ is a bridge in $ G $, or
				\item there exists a vertex set $ S=S(e)\subseteq V(G)$ with
				$$ w(G-S)\leq \frac{\lvert S\rvert}{t} ~\emph{and}~ w\left((G-e)-S\right)>\frac{\lvert S\rvert}{t} ,$$
				and the edge $ e $ is a bridge in $ G-S $.
			\end{enumerate} 
			In the first case, we define $ S=S(e)=\emptyset $.
		\end{lem}
		\par

		Let $ G $ be a minimally $ t $-tough graph.  For readability of the proofs of following lemmas and Theorem \ref{result 12}  , we denote the follwing notations. 
		Let $ S(e) $ be a minimum vertex set guaranteed by Lemma \ref{result 1} for every edge $ e=uv\in E(G) $, and $ \big|S(e)\big|=k(e) $. Let $ C(e) $ be the component of $ G-S(e) $ containing $ e $ and $ D(e) $ be the union of components of $ \big(G-S(e)\big)-C(e) $. In particular, let $ C_{u}(e)$ and $C_{v}(e) $ denote the components of $ \left(G-e\right)-S(e) $ containing $ u $ and $ v $, respectively. 
		
		Lemma \ref{result 1} implies that the following Lemma holds.
		\begin{lem}\label{result 15}
			Let $ G $ be a minimally $ t $-tough graph. For an arbitrary edge $ e=uv\in E(G) $, if $ D(e)=\emptyset $, then $ \delta(G)\leq 2t $ or every endpoint of $ e $ is contained in a $ 2t $-vertex-cut of $ G $.
		\end{lem}
		
		\begin{proof}
		 Since $ D(e)=\emptyset $, we can obtain $ k(e)\leq 2t-1 $ by Lemma \ref{result 1}. Suppose that $ k(e)\leq 2t-2 $. Assume $ e=uv $. Since $ \delta(G)\geq 2\kappa(G)\geq 2t $, we have $ C_{u}\neq \left\{u\right\} $. Note that $ S(e)\cup \left\{u\right\}$ is a $ (2t-1) $-vertex-cut of $ G $, a contradiction. So $ k(e)=2t-1 $. Clearly, if $ C_{u}(e)=\left\{u\right\} $ or  $ C_{v}(e)=\left\{v\right\} $, then $\delta(G)\leq \min \left\{d_{G}(u), d_{G}(v)\right\}=2t $. If $ C_{u}(e)\neq \left\{u\right\} $ and $ C_{v}(e)\neq \left\{v\right\} $, then $ \left\{u\right\}\cup S(e) $ and $ \left\{v\right\}\cup S(e) $ are  $ 2t $-vertex-cuts of $ G $. 
		\end{proof}
		
		The following lemma gives more information of the vertices in $ N_{S(e)}\big(C(e)\big) .$
		\begin{lem}\label{result 2}
			Let $ G $ be a minimally $ t $-tough, claw-free graph. For an arbitrary edge $ e=uv\in E(G) $, let $ S_{1}=N_{S(e)}\big(C(e)\big)-N_{S(e)}\big(D(e)\big) $. If $ D(e)\neq \emptyset $, then the following statements hold. 
			\begin{enumerate}[(1)]
				\item If $ k(e)\geq \lceil 3t\rceil $, then $ \big|N_{S(e)}\big(C(e)\big)\cap N_{S(e)}(C_{i})\big|\leq 2t-1 $, where $ C_{i} $ is an arbitrary component of $ G-S(e)-C(e) $.\label{result 6}
				
				\item   $ \lvert S_{1}\rvert\leq \lceil t\rceil-1$ and $ 2t \leq \big|N_{S(e)}\big(C(e)\big)\big|\leq 4t-1-\lvert S_{1}\rvert .$
				Moreover, if $  \lvert S_{1}\rvert =4t-1-\big|N_{S(e)}\big(C(e)\big)\big| $, then every vertex of $ S(e)-S_{1} $ is contained in a $ 2t $-vertex-cut of $ G $.\label{result 3}
				
				\item If $\big|N_{S(e)}\big(C(e)\big)\big|=4t-1 $, then every vertex of $ S(e) $ is contained in a $ 2t $-vertex-cut of $ G $.
				
				\item If $ \big|N_{S(e)}\big(C(e)\big)\big|=4t-2 $, let $ C_{j} $ be an arbitrary component of $ G-S(e)-C(e) $, then $ \lvert N_{S(e)}\left(C_{j}\right)\rvert\leq 2t+1 $ and
				there are at least $ 2t-1 $ vertices of $ N_{S(e)}\big(C(e)\big) $ contained in some $ 2t $-vertex-cuts of $ G $. \label{result 11}
				
				\item Assume $ C_{u}(e)\neq \left\{u\right\} $ and $ C_{v}(e)\neq \left\{v\right\} $.  If there exists a vertex $ w\in S(e) $ such that $ N_{C(e)}(w)=\left\{u, v\right\} $, then every vertex of $ \left\{u, v\right\} $ is contained in a $ 2t $-vertex-cut of $ G $. Otherwise every vertex of $ \left\{u\right\}\cup S_{1} $ or $ \left\{v\right\}\cup S_{1} $ is contained in a $ 2t $-vertex-cut of $ G $.  \label{result 4}
				
				\item If $ C_{u}(e)=\left\{u\right\} $ and $ C_{v}(e)\neq \left\{v\right\} $, then $ d_{G}(u)\leq 2t+1 $ or  $ v$ is contained in a $ 2t $-vertex-cut of $ G $.  \label{result 5}
				
			\end{enumerate}
		\end{lem}
		
		{\bf Proof.}  Let $ l=w\big(G-S(e)\big) $ and $ C_{1}=C(e), C_{2}, \cdots, C_{l} $ be the components of $ G-S(e) $. By Lemma \ref{result 1}, we have $ k(e)<t\cdot l+t.$
		\begin{enumerate}[(1)]
			\item Let $ C_{i} $ be an arbitrary component of $ G-S(e)-C(e) $ and $ W=\big|N_{S(e)}\big(C(e)\big)\cap N_{S(e)}(C_{i})\big| $. Suppose to the contrary that $ \lvert  W\rvert \geq 2t $.  Clearly, as $ \lceil 3t\rceil\leq k(e)< t\cdot l+t$, we can obtain $ w\big(G-S(e)\big)\geq 3 $. Since $ G $ is claw-free, every vertex of $ W $ is not adjecent to $ G(e)-S(e)-C(e)-C_{i} $. Thus $ w\Big(G-\big(S(e)-W\big)\Big)=l-1\geq 2 $. By the toughness of $ G $, we have
			$$ \big|S(e)-W\big|\geq t\cdot\left(l-1\right) .$$ By $ \lvert  W\rvert \geq 2t $, we have $ k(e)\geq t\cdot l+t $, a contradiction. 
			
			\item By the toughness of $ G $, we know that $ \lvert S(e)-S_{1}\rvert\geq t\cdot l .$ Since $ k(e)< t\cdot l+t $, we can obtain $ \lvert S_{1}\rvert\leq \lceil t\rceil-1. $  
			
			As $ \kappa(G)= 2t$, we can obtain $ \big|N_{S(e)}(C_{i})\big|\geq 2t $ for each $ i\in\left\{1, 2, \cdots, l \right\} $. Thus there are at least $ 2t\cdot (l-1)+\big|N_{S(e)}\big(C(e)\big)\big| $ edges coming from the components of $ G-S(e) $ to $ S(e) $ counting at most one from any component $ C_{i} $ to a particular vertex of $ S(e) $. Since $ G $ is claw-free, every vertex of $ S(e)-S_{1} $ have neighbors in at most two components of $ G-S(e) $. Then there are at most 2$\big(k(e)-\lvert S_{1}\rvert \big)+\lvert S_{1}\rvert $ edges coming from $ S(e) $ to the components of $ G-S(e) $ counting at most one edge from every component to a particular vertex of $ S(e) $. Therefore we conclude 
			\begin{equation}
				\big|N_{S(e)}\big(C(e)\big)\big|+2t\cdot (l-1)\leq 2\left(k(e)-\lvert S_{1}\rvert \right)+\lvert S_{1}\rvert <2\left(t\cdot l+t\right)- \lvert S_{1}\rvert
			\end{equation}
			and then $ 2t\leq\big|N_{S(e)}\big(C(e)\big)\big|\leq 4t-1-\lvert S_{1}\rvert $. 
			
			If $ \lvert S_{1}\rvert=4t-1- \big|N_{S(e)}\big(C(e)\big)\big|$, by (1), then $$ \big|N_{S(e)}\big(C(e)\big)\big|+2t\cdot (l-1)=2\left(k(e)-\lvert S_{1}\rvert \right)+\lvert S_{1}\rvert .$$ Therefore every component of $ G-S(e)-C(e) $ has $ 2t $ neighbors in $ S(e) $ and every vertex of $ S(e)-S_{1} $ has neighbors in exactly two components. Thus every vertex of $ S(e)-S_{1} $ is contained in a $ 2t $-vertex-cut of $ G $.
			
			\item If $ \big|N_{S(e)}\big(C(e)\big)\big|=4t-1 $, by a similar argument as above, then we have   
			$$ 4t-1+2t\cdot (l-1)\leq 2k(e)<2\left(t\cdot l+t\right) ,$$ so $ 4t-1+2t\cdot (l-1)=2k(e) $. It follows that every component of $ G-S(e)-C(e) $ has $ 2t $ neighbors in $ S(e) $ and every vertex of $ S(e) $ has neighbors in exactly two components. Therefore every vertex of $ S(e) $ is contained in a $ 2t $-vertex-cut of $ G $.  
			
			\item Let $ C_{2} $ be the component of $ G-S(e)-C(e) $, which has the most neighbors in $ S(e) $. By a similar argument as Lemma \ref{result 2} (2), we have  
			\begin{equation}
				\big|N_{S(e)}\big(C(e)\big)\big|+\big|N_{S(e)}\left(C_{2}\right)\big|+2t\cdot (l-2)\leq 2k(e)<2\left(t\cdot l+t\right) .
			\end{equation}
			As $ \big|N_{S(e)}\big(C(e)\big)\big|=4t-2 $, we can obtain $ \big|N_{S(e)}\left(C_{2}\right)\big|\leq 2t+1 $. By $ \kappa(G)=2t $, we conclude $ \big|N_{S(e)}\left(C_{i}\right)\big|\geq 2t $ for $ 2\leq i\leq l $. 
			
			If $ \big|N_{S(e)}\left(C_{2}\right)\big|=2t $, since $ \lvert N_{S(e)}\left(C_{2}\right)\rvert \geq \lvert N_{S(e)}\left(C_{i}\right)\rvert $ for $ 3\leq i\leq l $, then every component of $ G-S(e)-C(e) $ has exactly $ 2t $ neighbors in $ S(e) $. Thus  
			there are exactly $ 2t\cdot l+2t-2 $ edges coming from the components of $ G-S(e)-C(e) $ to $ S(e) $ counting at most one from any component $ C_{i} $ to a particular vertex of $ S(e) $. Since $ G $ is claw-free, every vertex of $ S(e) $ have neighbors in at most two components of $ G-S(e) $. And every vertex of $ S(e) $ is adjacent to at least one component of $ G-S(e) $, otherwise $ t\leq \dfrac{k(e)-1}{w\big(G-S(e)\big)+1} $, which contradicts Lemma 2.2. So there are exactly 2$\big(k(e)-x_{1}\big)+x_{1}  $ edges coming from $ S(e) $ to the components of $ G-S(e) $ counting at most one edge from every component to a particular vertex of $ S(e) $, where $ x_{1} $ is the number of the vertices in $ S(e) $ adjacent to exactly one component of $ G-S(e) $. Therefore 
			$$ 2t\cdot l+2t-2= 2\left(k(e)-x_{1} \right)+x_{1}\leq 2t\cdot l+2t-1-x_{1}.$$ Clearly, $ 0\leq x_{1}\leq 1. $ 
			If $ x_{1}=0 $, then $ 2t\cdot l+2t-2=2k(e) $, so every vertex $ S(e) $ is  adjacent to exactly two components and every component of $ S(e) $ has exactly $ 2t $ neighbors. Hence every vertex of $ G-S(e)-C(e) $ is conained in a $ 2t $-vertex-cut of $ G $.		
			If $ x_{1}=1 $, then $ 2t\cdot l+2t-2= 2(k(e)-1)+1 $. Let $ w $ be the vertex in $ S(e) $ adjacent to only one component of $ G-S(e) $. So every vertex $ S(e)-\left\{w\right\} $ has neighbors in exactly two components and every component of $ G-S(e)-C(e) $ has exactly $ 2t $ neighbors. Hence every vertex of $  S(e)-\left\{w\right\} $ is contained in a $ 2t $-vertex-cut of $ G $. 
			
			If $ \big|N_{S(e)}\left(C_{2}\right)\big|=2t+1 $, by (2), then $$ \big|N_{S(e)}\big(C(e)\big)\big|+\big|N_{S(e)}\left(C_{2}\right)\big|+2t\cdot (l-2)= 2k(e)=2t\cdot l+2t-1 .$$ Note that every component of $ G-S(e)-C(e)-C_{2} $ has $ 2t $ neighbors in $ S(e) $ and every vertex of $ S(e) $ has neighbors in exactly two components of $ G-S(e) $. Thus every vertex of $ N_{S(e)}\big(C(e)\big) -N_{S(e)}\left(C_{2}\right)$ is contained a $ 2t $-vertex-cut of $ G $. And by Lemma \ref{result 2} (\ref{result 6}), we have $ \big|N_{S(e)}\big(C(e)\big)\cap N_{S(e)}\left(C_{2}\right)\big|\leq 2t-1$. It follows that $ \big|N_{S(e)}\big(C(e)\big)- N_{S(e)}\left(C_{2}\right)\big|\geq 2t-1$ from $ \big|N_{S(e)}\big(C(e)\big)\big|=4t-2$. Thus the statement holds.

			\item Let $$ A_{1}=N_{S(e)}\big(C_{u}(e)-\left\{u\right\}\big)\cap N_{S(e)}\big(C_{v}(e)\big),$$
			$$ A_{2}=N_{S(e)}\big(C_{u}(e)\big)\cap N_{S(e)}\big(C_{v}(e)-\left\{v\right\}\big) ,$$ 
			$$A_{3}=\left\{w:w\in N_{S(e)}\big(C(e)\big)-N_{S(e)}\big(D(e)\big)~\emph{and}~ N_{C(e)}(w)=\left\{u, v\right\}\right\},$$
			$$ B_{1}=N_{S(e)}\big(C_{u}(e)-\left\{u\right\}\big)\cap N_{S(e)}\big(D(e)\big),$$ $$ B_{2}=N_{S(e)}\big(C_{v}(e)-\left\{v\right\}\big)\cap N_{S(e)}\big(D(e)\big), $$ $$
			B_{3}=N_{S(e)}(u)\cap N_{S(e)}(v)\cap N_{S(e)}\big(D(e)\big),$$
			Without loss of generality, we may assume $ \lvert A_{1}\rvert\geq \lvert A_{2}\rvert $.	
			
			Since $ G $ is claw-free, we have $ \left(A_{1}\cup A_{2}\right)\cap N_{S(e)}\big(D(e)\big)=\emptyset $ and $ B_{i}\cap B_{j}=\emptyset $ for $ i, j\in \left\{1, 2, 3\right\} $ and $ i\neq j .$ So  $ \left(A_{1}\cup A_{2}\right)\cap (A_{3}\cup B_{i})=\emptyset $ for $ i\in \left\{1, 2, 3\right\} $. 
			By Lemma \ref{result 2} (\ref{result 3}), we can obtain
			\begin{equation}
				\lvert A_{1}\cup A_{2}\rvert +\lvert A_{3}\rvert+\sum_{j=1}^{3}\lvert B_{i}\rvert\leq \big|N_{S(e)}\big(C(e)\big)\big|\leq 4t-1-\big(\lvert A_{1}\cup A_{2}\rvert +\lvert A_{3}\rvert\big).
			\end{equation} 
			Furthermore, by the minimality of $ S(e) $, we know 
			$ N_{S(e)}\big(C_{u}(e)-\left\{u\right\}\big)=A_{1}\cup B_{1} $ and $ N_{S(e)}\big(C_{v}(e)-\left\{v\right\}\big)=A_{2}\cup B_{2} .$  For $ i\in\left\{1, 2\right\}$, as $ \kappa(G)=2t $ by Theorem \ref{result 8}, we have 
			$\lvert A_{i}\rvert +\lvert B_{i}\rvert \geq 2t-1$. And by (3), we have 
			\begin{equation}
				4t-2+2\lvert A_{3}\rvert+\lvert B_{3}\rvert\leq 2\lvert A_{1}\rvert +2\lvert A_{3}\rvert+\sum_{j=1}^{3}\lvert B_{i}\rvert\leq 4t-1.
			\end{equation}
			Thus $2\lvert A_{3}\rvert+\lvert B_{3}\rvert \leq 1$. Clearly, $ \lvert A_{3}\rvert=0$.	
			
			Firstly, consider the case when $ \lvert B_{3}\rvert = 1 $. From (4), we can obtain $ \lvert  A_{1}\rvert+\lvert B_{1}\rvert+\lvert A_{2}\rvert+\lvert B_{2}\rvert=4t-2 $. So $\lvert A_{i}\rvert +\lvert B_{i}\rvert \geq 2t-1$ for $ i\in \left\{1, 2\right\} $. It follws that $ \left\{u\right\}\cup A_{1}\cup B_{1} $ and $ \left\{v\right\}\cup A_{2}\cup B_{2} $ are $ 2t $-vertex-cuts of $ G $.
			
			Next consider the case when $\lvert B_{3}\rvert =0 $.  We claim that $ A_{2}\subseteq A_{1}=S_{1} $. Suppose to the contrary that $ A_{2}-A_{1}\neq \emptyset $. Then $ \lvert A_{1}\cup A_{2}\rvert \geq \lvert A_{1}\rvert +1$. Since $ \lvert A_{i}\rvert +\lvert B_{i}\rvert \geq 2t-1 $ for $ i\in\left\{1, 2\right\} $, by (3), we can obtain $$ 4t\leq 2\lvert A_{1}\rvert +2+\lvert B_{1}\rvert +\lvert B_{2}\rvert\leq  \big|N_{S(e)}\big(C(e)\big)\big|\leq 4t-1,$$ a contradiction. Thus 
			\begin{equation}
				2\lvert A_{1}\rvert +\lvert B_{1}\rvert +\lvert B_{2}\rvert\leq 4t-1 .
			\end{equation}
			
			If $ \lvert A_{2}\rvert <\lvert A_{1}\rvert $, by (5), then $ \lvert A_{1}\rvert +  \lvert A_{2}\rvert +\lvert B_{1}\rvert+ \lvert B_{2}\rvert<4t-1$. Hence we can obtain $ \lvert  A_{1}\rvert+\lvert B_{1}\rvert=2t-1=\lvert A_{2}\rvert+\lvert B_{2}\rvert$. It follows that $ \left\{u\right\}\cup S_{1}\cup B_{1}$ and $ \left\{v\right\}\cup A_{2}\cup B_{2}$ are $ 2t $-vertex-cuts of $ G $. If $ S_{1}=A_{1}=A_{2} $, by (5), then $ \lvert  S_{1}\rvert+\lvert B_{1}\rvert=2t-1$ or $\lvert S_{2}\rvert+\lvert B_{2}\rvert=2t-1.$ So $ \left\{u\right\}\cup S_{1}\cup B_{1}$ or $ \left\{v\right\}\cup S_{1}\cup B_{2}$ is $ 2t $-vertex-cut of $ G $.

			\item 
			Let $ M=N_{S(e)}(u)\cap N_{S(e)}\big(C_{v}(e)-\left\{v\right\}\big) $. Since $ G $ is claw-free, we have $ M\cap N_{S(e)}\big(D(e)\big)=\emptyset $. By Lemma \ref{result 2} (\ref{result 3}), $2t\leq  \big|N_{S(e)}\big(C(e)\big)\big|\leq 4t-1-\lvert M\rvert $. And as $ \kappa(G)=2t $, we have $ \big|N_{S(e)}\big(C_{v}(e)-\left\{v\right\}\big)\big|\geq 2t-1 $. If $ M=\emptyset $, then $ \big|N_{S(e)}(u)\big|=\big|N_{S(e)}\big(C(e)\big)\big|-\big|N_{S(e)}\big(C(e)-\left\{v\right\}\big)\big|\leq 2t $, so $ d_{G}(u)\leq 2t+1 $. If $ M\neq \emptyset $, let $ x_{1}=\big|N_{S(e)}\big(C_{v}(e)-\left\{v\right\}\big)\cap N_{S(e)}\big(D(e)\big)\big| $ and $ x_{2}=\big|N_{S(e)}\left(u\right)\cap N_{S(e)}\big(D(e)\big)\big|$, then $ x_{1}\leq 2t-1-M $ or $ x_{2}\leq 2t-1-M $. Suppose to the contrary that $ x_{1}\geq 2t-\lvert M\rvert $ and $ x_{2}\geq 2t-\lvert M\rvert .$ Then $ x_{1}+x_{2}+2\lvert M\rvert\geq 4t $. Moreover, since $ M\cap N_{S(e)}\big(D(e)\big)=\emptyset $, we can obtain $$ 4t-\lvert M \rvert\leq x_{1}+x_{2}+\lvert M\rvert\leq\big|N_{S(e)}\big(C_{v}(e)\big)\big|\leq  4t-1-\lvert M\rvert  ,$$ a contradiction. By the minimality of $ S(e) $, 	we see that $\big|N_{S(e)}\big(C_{v}(e)-\left\{v\right\}\big)\big|=\lvert M\rvert +x_{1}\leq 2t-1 $ or $\big|N_{S(e)}\left(u\right)\big|=\lvert M\rvert+x_{2}\leq 2t-1 $. As $ \kappa(G)=2t $, we can obtain $ d_{G}(u)=2t $ or  $ \big|N_{S(e)}\big(C_{v}(e)-\left\{v\right\}\big)\big|= 2t-1 $.

		\end{enumerate}

		In some cases, there exist many vertices of $ N_{S(e)}\left(C(e)\right) $ contained in some minimum cut-sets of $ G $ by Lemma \ref{result 2}. Using the following Theorem, we can give an upper bound on the minimum degree of minimally $ t $-tough, claw-free graphs for $ t\geq 2 $.

		\begin{thm}[\cite{Mad}]\label{result 7}
			Let $ G $ be a connected graph, $ A $ be an atom of $ G $ and $ T $ be a minimum cut-set of $ G $. If $ A\cap T\neq \emptyset $, then $ A\subseteq T $ and $ \left|A\right|\leq \dfrac{1}{2}\kappa(G) $.
		\end{thm}
		\noindent{\bf Proof of Theorem 1.8.}
		Let $ G $ be a minimally $ t $-tough, claw-free geaph. If $ \delta(G)\leq \lceil 3t\rceil -1 $, then Theorem 1.8 holds. Hence we assume $ \delta(G)\geq \lceil 3t\rceil  $ in the following. 	
		
		Let  $ A $ be an atom of  $ G $ and $ T=N_{G}(A) $. By Lemma \ref{result 8}, $ \lvert T\rvert=\kappa(G)=2t $. Clearly, $ \lvert A\rvert\geq 2 $, otherwise $ \delta(G)=2t $, which contradicts our assumption. 
		\begin{adjustwidth}{2em}{0cm}
			{\bf Claim 1.}\label{result 16}
			Every vertex of $ A $ is not contained in a $ 2t $-vertex-cut of $ G $. 
			\begin{proof}
				Suppose to the contrary that there is a vetex of $ A $ contained in a $ 2t $-vertex-cut of $ G $.  By Theorem 
				\ref{result 7}, we can obtain $ \delta(G)\leq \lceil 3t\rceil -1 $, a contradiction. 
			\end{proof}
		\end{adjustwidth}
		\begin{adjustwidth}{2em}{0cm}
			{\bf Claim 2.}
			For an arbitrary edge $ e=xy $ in $ A $, we have $ D(e)\neq \emptyset $, $ C(e)=\left\{x, y\right\} $ and $  \lceil 3t\rceil-1\leq\big|N_{S(e)}\big(C(e)\big)\big|\leq 4t-3 $ 
			\begin{proof}
				Suppose to the contrary that $ D(e)=\emptyset $. By Lemma \ref{result 15}, as $ \delta(G)\geq \lceil 3t\rceil $, we know that $ x $ is contained in a $ 2t $-vertex-cut of $ G $, which contradicts Claim 1. Thus $ D(e)\neq \emptyset .$
				
				Suppose to the contrary that $ C(e)\neq \left\{x, y\right\} $. By Lemma \ref{result 2} (\ref{result 4}) and (\ref{result 5}), as $ \delta(G)\geq \lceil 3t\rceil$, we note that $ x $ or $ y $ is contained in a $ 2t $-vertex-cut of $ G $, which contradicts Claim 1. Thus $ C(e)=\left\{x, y\right\} $ and then $  \big|N_{S(e)}\big(C(e)\big)\big|\geq \lceil 3t\rceil-1 $.

				Suppose that $ \big|N_{S(e)}\big(C(e)\big)\big|\geq 4t-2 $. 
				Since $ C(e)=\left\{x, y\right\} $, we see that $ N_{S(e)}\big(C(e)\big)= N_{G}(x)\cup N_{G}(y)-\left\{x, y\right\} $. Let $ S_{1}=N_{S(e)}\big(C(e)\big)-T $. Clealy, $ S_{1}\subseteq A $. As $ \lvert T\rvert=2t $, we have $  \lvert S_{1}\rvert\geq 2t-2\geq 2$. And by Lemma \ref{result 2} (\ref{result 3}), we can obtain $ \big|N_{S(e)}\big(C(e)\big)\big|\in \left\{4t-2, 4t-1\right\}$.  
				
				{\bf Case 1.} 
				$ \big| N_{S(e)}\big(C(e)\big) \big|=4t-1$.
				
				By Lemma \ref{result 2} (\ref{result 3}), we know that every vertex of $S_{1} $ is contained in a $ 2t $-vertex-cut of $ G $, which contradicts Claim 1.
				
				{\bf Case 2.} 
				$ \big| N_{S(e)}\big(C(e)\big) \big|=4t-2$.
				
				Let $ W $ denote the set of vertices of $ N_{S(e)}\big(C(e)\big) $ contained in some $ 2t $-vertex-cuts of $ G $. By Lemma \ref{result 2} (\ref{result 11}), we have $ \lvert W\rvert\geq 2t-1 $. We claim that $ W\subseteq T $. Otherwise, since $ W\subseteq N_{G}(x)\cup N_{G}(y) $, we see that $ W\cap A\neq \emptyset $, which contradicts Claim 1.  Since $ T  $ is the minimum cut set of $ G $, then every vertex of $ W $ is adjacent to some vertex in $ G-T-A $. As $ W\subseteq N_{G}(x)\cup N_{G}(y)$ and $ G $ is claw-free, we see that $ N_{A}(W)\subseteq N_{G}(x)\cup N_{G}(y). $
				
				Suppose $ A-N_{G}[x]\cup N_{G}[y]\neq \emptyset $. As $ W\subseteq T\cap N_{S(e)}\big(C(e)\big) $ and $ \lvert W\rvert\geq 2t-1 $, we can obtain $ \lvert S_{1}\rvert \leq \big| N_{S(e)}\big(C(e)\big)\big|-\lvert W\rvert\leq 2t-1 $. Since $ N_{A}(W)\subseteq N_{G}(x)\cup N_{G}(y) $ and  $ \lvert W\rvert\geq 2t-1 $, we see that  $ S_{1}\cup (T-W) $ is a cut set of $ G $ of order at most $ 2t $, which contradicts Claim 1. Thus $ A-N_{G}[x]\cup N_{G}[y]=\emptyset $, i.e. $ A=\left\{x, y\right\}\cup S_{1} $. 
				
				Hence we have $ N_{D(e)}\left(S_{1}\right)\subseteq T $. Since $ \lvert S_{1}\rvert\geq 2 $, by Lemma \ref{result 2} (\ref{result 3}), we know that $ N_{D(e)}(S_{1})\neq \emptyset $.  Moreover, as $ \lvert W\rvert \geq 2t-1 $ and $ W\subseteq T $, we can obtain $\big|N_{D(e)}\left(S_{1}\right)\big|=1$ and $ \lvert W\rvert= 2t-1 .$  Suppose that there exists a vertex $ w_{i}\subseteq S_{1} $ not adjacent to $ D(e) $, by Lemma \ref{result 2} (\ref{result 3}), then every vertex of $ S_{1}-\left\{w_{i}\right\} $ is contained in a $ 2t $-vertex-cut of $ G $, which contradicts Claim 1. Thus every vertex of $ S_{1}$ has exactly one neighbor in $ D(e) $. Let $ N_{D(e)}\left(S_{1}\right)=\left\{t_{i} \right\}$ and $ C_{i} $ be the component containing $ t_{i} $ of $ G-S(e)-C(e) $. By Lemma \ref{result 2} (\ref{result 11}), we have $ \big|N_{S(e)}(C_{i})\big|\leq 2t+1 $. We claim that $ C_{i}=\left\{t_{i}\right\} $. Otherwise, as $ N_{D(e)}\left(S_{1}\right)=\left\{t_{i}\right\} $, it is easy to see that $ \left\{t_{i}\right\}\cup \big(N_{S(e)}(C_{i})-S_{1}\big) $ is a cut-set of $ G $. As $ \lvert S_{1}\rvert=\big|N_{S(e)}\big(C(e)\big)\big|-W= 2t-1 $, we can obtain $ \big|\left\{t_{i}\right\}\cup \big(N_{S(e)}(C_{i})-S_{1}\big)\big|\leq 3 $, which contradicts $ \kappa(G)\geq 2t\geq 4 $, a contradiction. So $ d_{G}(t_{i})\leq 2t+1 $, a contradiction. 
				
				Therefore $  \lceil 3t\rceil-1\leq k(e)\leq 4t-3 .$ 
			\end{proof}
		\end{adjustwidth}
		
		Assume $ e_{1}=uv\in E(A) $ satisfies $ \big|N_{G}(v)-N_{G}(u)\big|=\max\left\{\big|N_{G}(y)-N_{G}(x)\big|: xy\in E(A)\right\} .$  By Claim 2, we know that $ D(e_{1})\neq \emptyset $, $ C(e_{1})=\left\{u, v\right\} $ and $ \lceil 3t\rceil-1\leq k(e_{1})\leq 4t-3 $. Hence $ N_{S(e_{1})}\big(C(e_{1})\big)\subseteq N_{G}(u)\cup N_{G}(v) $. Let $ S_{2}=N_{S(e_{1})}\big(C(e_{1})\big)-T $. Clearly,  $ \emptyset\neq S_{2}\subseteq A $.
		
		{\bf Case 1.} 
		$ \big| N_{S(e_{1})}\big(C(e_{1})\big) \big|=\lceil 3t\rceil-1$.

		By the minimality of $ S(e_{1}) $, we know that $ S(e_{1})=N_{S(e_{1})}\big(C(e_{1})\big) .$ Since $ C(e)=\left\{u, v\right\} $ and $ \delta(G)\geq \lceil 3t\rceil $, we have $ S(e_{1})\subseteq N_{G}(u)\cap N_{G}(v).$ By the maximality of $ \big|N_{G}(v)-N_{G}(u)\big| $, we see that $ N_{G}(S_{2})\subseteq N_{G}[u] $, so $ S_{2}\subseteq N_{S(e_{1})}\big(C(e_{1})\big)-D(e_{1}) $. By Lemma \ref{result 2} (\ref{result 3}), we have $ \lvert S_{2}\rvert\leq \lceil t\rceil -1 $. And as $ \lvert S_{2}\rvert \geq k(e_{1})-\lvert T\rvert=\lceil t\rceil -1$, we can obtain $ \lvert S_{2}\rvert=\lceil t\rceil -1$. In addition, since $ \delta(G)\geq \lceil 3t\rceil $, we see that $ N_{G}(w_{i})=S(e_{1})\cup \left\{u, v\right\}-\left\{w_{i}\right\} $ for each $ w_{i}\in S_{2} .$ 
		
		Let $ w\in S(e_{1})\cap T $ and $ e_{2}=uw\in E(G) $. Since $ \left\{v\right\}\cup S_{2}\subseteq N_{G}(u)\cap N_{G}(w) $, we see that $ \left\{v\right\}\cup S_{2}\subseteq N_{S(e_{2})}\big(C(e_{2})\big) $. As $ \left\{v\right\}\cup S_{2}\subseteq N_{G}(u) $, we know that $ N_{D(e_{2})}\big(\left\{v\right\}\cup S_{2}\big)=\emptyset .$  By Lemma \ref{result 2} (\ref{result 3}), since $ \lvert \left\{v\right\}\cup S_{2}\rvert =\lceil t\rceil $, we have $ D(e_{2})=\emptyset .$  Thus we can obtain $ k(e_{2})\leq 2t-1 $ by Lemma 2.2.  We claim that $ C_{w}(e_{2})=\left\{w\right\} $. Suppose to the contrary that  $ C_{w}(e_{2})\neq\left\{w\right\} $. As $ \left\{v\right\}\cup S_{2}\subseteq N_{G}(u) $, we see that $ N_{S(e_{2})}\big(C(e_{2})-\left\{w\right\}\big)\cap \big(\left\{v\right\}\cup S_{2}\big)=\emptyset .$  As $  k(e_{2})\leq 2t-1 $ and $ \left|\left\{v\right\}\cup S_{2}\right|=\lceil t\rceil  $, we know that $ \big|N_{S(e_{2})}\big(C(e_{2})-\left\{w\right\}\big)\big|\leq \lceil t\rceil-1 $. So $ \left\{w\right\}\cup N_{S(e_{2})}\big(C(e_{2})-\left\{w\right\}\big) $ is a $ \lceil t\rceil  $-vertex-cut of $ G $, which contradicts $ \kappa(G)=2t $. Thus $ d_{G}(w)\leq 2t ,$ which contradicts our assumption. 
		
		{\bf Case 2.}
		$ \lceil  3t\rceil \leq \big| N_{S(e_{1})}\big(C(e_{1})\big) \big|\leq 4t-3$.
		
		Let $ x=\big|\left\{u\right\}\cup \big(N_{S(e_{1})}(v)-N_{S(e_{1})}(u)\big)\big|=\big|N_{G}(v)-N_{G}(u)\big| .$ As $ C(e)=\left\{u, v\right\} $, we have $ \delta(G)\leq d_{G}(u)\leq 4t-2-(x-1) $.
		
		Clearly, $ \lvert S_{2}\rvert \geq \lceil t\rceil .$ By Lemma \ref{result 2} (\ref{result 3}), there exists a vertex $ w_{1} $ in $ S_{2} $ adajcent to $ D(e_{1}) $.  Without loss of generality, we assume $ uw_{1}\in E(A) $ and $ e_{3}=w_{1}t_{1}\in E(G)$, where $ t_{1}\in D(e_{1}) .$ Let $ z_{1}=\big|N_{S(e_{1})}(w_{1})- N_{S(e_{1})}\big(C(e_{1})\big)\big|$.  As $$\Big|\left\{u\right\}\cup N_{D(e_{1})}(w_{1})\cup \Big(N_{S(e_{1})}(w_{1})-N_{S(e_{1})}\big(C(e_{1})\big)\Big)\Big|\leq \big|N_{G}(w_{1})-N_{G}(u)\big|\leq \big|N_{G}(v)-N_{G}(u)\big|=x ,$$ we see that $ \big|N_{D(e_{1})}\left(w_{1}\right)\big|\leq x-z_{1}-1 $. Moreover, by Lemma \ref{result 2} (\ref{result 6}), we know that $ \big|N_{S(e_{1})}(t_{1})\cap N_{S(e_{1})}\big(C(e_{1})\big)\big|\leq 2t-1 .$ It follows that 
		\begin{equation}
			\begin{aligned}
				\big|N_{G}(w_{1})\cap N_{G}(t_{1})\big|&\leq\big|N_{S(e_{1})}(t_{1})\cap N_{S(e_{1})}\big(C(e_{1})\big)-\left\{w_{1}\right\}\big|+ \big|N_{S(e_{1})}(w_{1})- N_{S(e_{1})}\big(C(e_{1})\big)\big|\\&+\big|N_{D(e_{1})}(w_{1})-\left\{t_{1}\right\}\big|\\&\leq 2t-2+z_{1}+x-z_{1}-2=2t+x-4, \nonumber		
			\end{aligned}
		\end{equation}

		{\bf Case 2.1.}
		$ t_{1}\in  A $.
		
		By Claim 1, we know that $ C(e_{3})=\left\{w_{1}, t_{1}\right\} $ and $ k(e_{3})\leq 4t-3 .$ Clearly, $ N_{G}(w_{1})\cap N_{G}(t_{1})\subseteq S(e_{3}) $. Thus  
		
		\begin{equation}
			\begin{aligned}
				\delta(G)&\leq \dfrac{k(e_{3})-\lvert N_{G}(w_{1})\cap N_{G}(t_{1}) \rvert}{2}+\lvert N_{G}(w_{1})\cap N_{G}(t_{1})\rvert +1\\&\leq  \dfrac{4t-3-(2t+x-4)}{2}+2t-3+x=3t+\dfrac{x-5}{2} .
			\end{aligned}
			\nonumber	
		\end{equation}
		Therefore we conclude $ \delta(G)\leq \min\left\{4t-1-x, 3t+\frac{x-5}{2}\right\}\leq \lceil \frac{10t}{3}\rceil-2$ .
		
		{\bf Case 2.2.}
		$ t_{1}\in  T $.
		
		Since $ A $ is an atom of $ G $, then $ t_{1} $ is adjacent to at least two vertices in $ A $. Moreover, since $ t_{1}  $ is adjacent some vertex of $ G-T-A $ and $ G $ is claw-free, we see that $ N_{A}(t_{1})\subseteq N_{G}(w_{1}) $. Let $ a\in N_{A}(w_{1})\cap N_{A}(t_{1}) $. Clearly, $ a\in S(e_{3}) $. Note that $ D(e_{3})\neq \emptyset $. Otherwise, by Lemma 2.3, since $ \delta(G)\geq \lceil 3t\rceil $, we know that $ w_{1} $ is contained in a $ 2t $-vertex-cut of $ G $, which contradicts Claim 1. And  $ C(e_{3})=\left\{w_{1}, t_{1} \right\} $.
		Otherwise, by Lemma \ref{result 2} (\ref{result 4}) and (\ref{result 5}), since $ \delta(G)\geq \lceil 3t\rceil $, we know that $ w_{1} $ or $ a $ is contained in a $ 2t $-vertex-cut of $ G $, which contradicts Claim 1.
		By Lemma \ref{result 2} (\ref{result 3}),  we know that $ k(e_{3})\leq 4t-1 $. If $ k(e_{3})= 4t-1 ,$ by Lemma \ref{result 2} (3), then $ a $ is contained in a $ 2t $-vertex-cut of $ G $, which contradicts Claim 1. If $ k(e_{3})\leq 4t-2 $, as $ C(e_{3})=\left\{w_{1}, t_{1} \right\} $, then $$ \delta(G)\leq \dfrac{4t-2-(2t+x-4)}{2}+2t-3+x=3t+\dfrac{x-4}{2} .$$ Thus $ \delta(G)\leq \min\left\{4t-1-x, 3t+\frac{x-4}{2}\right\}\leq \big \lceil \frac{10t-5}{3}\big \rceil .$
		
		Above all, we conclude $ \delta(G)\leq \max \left\{ \big \lceil \frac{10t}{3}\big \rceil -2 , \big\lceil \frac{10t-5}{3}\big \rceil \right\}\leq\big \lceil \frac{10t-5}{3}\big \rceil $.

	\end{document}